\documentclass[reqno,11pt]{amsart}
\usepackage[margin=1.5in]{geometry}
\usepackage{mathrsfs} 
\usepackage{verbatim,amsmath,amsthm,amsfonts,amssymb,latexsym,graphicx,mathtools,extpfeil,color,mathabx}
\usepackage{xypic}
\usepackage{enumerate}
\usepackage{amsmath}

\usepackage{color}
\newcommand{\blue}{\textcolor{blue}}

\makeindex
\def\R{{\Bbb R}}
\def\C{{\Bbb C}}
\def\Z{{\Bbb Z}}

\theoremstyle{theorem}
\newtheorem{thm}{Theorem}

\newtheorem{lem}[thm]{Lemma}

\theoremstyle{definition}

\newtheorem{rem}[thm]{Remark}

\begin{document}

\title[Double log smooth structure]{
Corrigendum of ``{\em  Construction of Kuranishi structures on the moduli spaces of pseudo holomorphic disks I}, Surveys in Differential Geometry XXII (2018), 133--190''}
\address{Simons Center for Geometry and Physics,
State University of New York, Stony Brook, NY 11794-3636 U.S.A.} 
\email{kfukaya@scgp.stonybrook.edu}
\address{Center for Geometry and Physics, Institute for Basic Sciences (IBS), Pohang, Korea \& Department of Mathematics,
POSTECH, Pohang, Korea} \email{yongoh1@postech.ac.kr}
\address{Graduate School of Mathematics,
Nagoya University, Nagoya, Japan} \email{ohta@math.nagoya-u.ac.jp}
\address{Research Institute for Mathematical Sciences, Kyoto University, Kyoto, Japan}
\email{ono@kurims.kyoto-u.ac.jp}
\author{Kenji Fukaya, Yong-Geun Oh, Hiroshi Ohta, Kaoru Ono}
\date{2024 February}

\maketitle

This note is the corrigendum of \cite[Lemma 9.1]{const1}. It concerns a coordinate system 
near the boundary and corners of a smooth structure of
 the compactified moduli space $\mathcal M^{\rm d}_{k+1,\ell}$ of stable marked disk with $k+1$ boundary marked points 
and $\ell$ interior marked point. 
{(See  \cite[Definition 2.3]{const1} for the definition of $\mathcal M^{\rm d}_{k+1,\ell}$.)}
The statement of the lemma has been also used in the applications appearing in
our several other related articles. This corrigendum will not affect 
them and exactly the same proofs as therein apply if one replaces 
the statement of \cite[Lemma 9.1]{const1} 
by that of Lemma \ref{lem91} below. (The authors thank A. Daemi and a referee of \cite{DF}, 
who pointed out this error.)

\section{Overview of the corrigendum}

Let us recall the notation we used in \cite[Lemma 3.4]{const1}.
We study the map
\begin{equation}\label{form9393}
\Phi : 
\prod_{a\in \mathcal A_{\bf p}^{\rm s} \cup \mathcal A_{\bf p}^{\rm d}} \mathcal V_a 
\times [0,c)^{m_{\rm d}} \times (D^2_{\circ}(c))^{m_{\rm d}}
\to \mathcal M^{\rm d}_{k+1,\ell}.
\end{equation}
Here
$\mathcal M^{\rm d}_{k+1,\ell}$ is the compactified moduli space of 
stable marked disk with $k+1$ boundary marked points and $\ell$ interior marked points:
\begin{itemize}
\item ${\bf p}$ denotes {an element of $\mathcal M^{\rm d}_{k+1,\ell}$}.
\item  $\mathcal A_{\bf p}^{\rm s}$ (resp. $\mathcal A_{\bf p}^{\rm d}$) is the set of 
irreducible components of ${\bf p}$ which is a sphere (resp. disk).
\end{itemize}
We denote by  {$\Sigma_a$} the irreducible component $a$ of ${\bf p}$ and use the following notations:
\begin{itemize}
\item ${\bf p}_a$ is {the marked disk or the marked sphere 
consisting of the irreducible component $\Sigma_a$ 
together with the marked points of {\bf p} on $\Sigma_a$ and  the nodal points of {\bf p} on $\Sigma_a$.}
\item 
$\mathcal V_a$ is a neighborhood of ${\bf p}_a$ in its associated moduli space.
\item 
$m_{\rm d}$ (resp. $m_{\rm s}$) is the number of the boundary nodes (resp. 
interior nodes) of ${\bf p}$. 
\item $[0,c)$  (resp. $D^2_{\circ}(c)$) 
is the parameter space to smoothen the boundary nodes (resp. interior nodes).
\end{itemize}

The map $\Phi_{s,\rho}$ is determined if we fix an
analytic family of coordinates of the nodal points in the sense of \cite[Definition 3.1]{const1}.
See \cite[Lemma 3.4]{const1}.

For $r_j \in [0,c)$ ($j=1,\dots, m_d$)
and  $\sigma_i \in D^2_0(c)$, ($i=1,\dots,m_s$), we define\footnote{We change the symbol $s$ used in \cite{const1} to $t$ here.
This is because we use $S$ in (\ref{form912}).} a coordinate system given by
\begin{equation}\label{form910}
\aligned
&T^{\rm d}_j = -{\rm log} r_j \in \R_{+} \cup \{\infty\},  \\
&T^{\rm s}_i = -{\rm log} \vert\sigma_i\vert\in \R_{+}\cup \{\infty\}, 
\quad \theta_i = -{\rm Im}({\log} \sigma_i) \in \R/2\pi\Z.
\endaligned
\end{equation}

\begin{equation}\label{form920}
t_j = 1/T^{\rm d}_j \in [0,-1/\log c), 
\quad
\rho_i =  \exp({\sqrt{-1}}\theta_i)/T^{\rm s}_j \in D^2(-1/\log c).
\end{equation}
Note that $\sigma_i = e^{-z}$ is a bounded holomorphic function of 
 the variable $z = T_i^{\rm s} + \sqrt{-1} \theta_i$ on the right half plane
 $\{z \in \C \mid \text{\rm Re} z > 0\}$.
\par

Composing  these coordinate changes with the map $\Phi$, we defined
\begin{equation}\label{form93930}
\Phi_{t,\rho} : 
\prod_{a\in \mathcal A_{\bf p}^{\rm s} \cup \mathcal A_{\bf p}^{\rm d}} \mathcal V_a 
\times [0,-1/\log c)^{m_{\rm d}} \times (D^2_{\circ}(-1/\log c))^{m_{\rm s}}
\to \mathcal M^{\rm d}_{k+1,\ell}.
\end{equation}
In  \cite[Lemma 9.1]{const1}
it was claimed that there is a unique smooth structure 
on $\mathcal M^{\rm d}_{k+1,\ell}$
such that {(\ref{form93930})} is a diffeomorphism.
There is an error in this statement. 
(It is true that there is a unique $C^1$ structure 
on $\mathcal M^{\rm d}_{k+1,\ell}$
such that {(\ref{form93930})} is a  diffeomorphism.)

To motivate and explain the idea of our current corrigendum, 
we start with explaining the reason why
the statement, as it is, of the aforementioned \cite[Lemma 9.1]{const1} 
does not hold for $\Phi_{t,\rho}$.

We consider the case when there is only one interior node and no boundary node.
Then there is only one gluing parameter which we denote by
$\sigma = e^{-(T + \sqrt{-1}\theta)}$.
When we change the coordinate of the gluing parameter to $\sigma' = \lambda \sigma$, 
where $\lambda$ is a positive real number,
the parameters $T$, $\theta$ change to
$
T' = T- \log \lambda$, $\theta' = \theta.
$
Then the coordinate $\rho = e^{\sqrt{-1}\theta}/T$ becomes 
$$
\rho' = e^{\sqrt{-1}\theta}/T' 
=  \frac{e^{\sqrt{-1}\theta}}{{-} \log \lambda + \frac{1}{\vert\rho\vert}}
= \frac{\rho}{1 {-}  (\log \lambda) \vert \rho\vert}.
$$
This function $\rho' = \rho'(\rho)$ is $C^1$ in $\rho$, but not $C^2$,  at the origin.
(This has been pointed out by  A. Daemi and a referee of \cite{DF}.)
\par
On the other hand, if we use the new coordinate 
system of $S = \log T$ and $\phi  = e^{\sqrt{-1}\theta}/S$, then we have
$$
\phi' = \frac{\phi}{1 + \vert \phi\vert  \log (1{-} (\log \lambda) e^{-1/\vert\phi\vert} )}.
$$
All the derivatives (of arbitrary order) of the denominator of the right hand side function
vanish at $\phi = 0$. In particular $\phi \mapsto \phi'$ is smooth at the origin.
\par
In the case where there is no interior node and one boundary node, 
we consider a real parameter $T$ and $t = 1/T$.
If we change $e^{-T}$ to $\lambda e^{-T}$, then $t$ is changed to $t'$, where
$$
t' = \frac{t}{1 {-}  (\log \lambda) t}.
$$
This is smooth with respect to the non-negative real parameter $t$ at $t=0$.
\begin{rem}
We can construct a $C^1$ smooth  structure of  $\mathcal M^{\rm d}_{k+1,\ell}$
still using $\Phi_{t,\rho}$. As far as we use multi-sections to perturb the moduli 
space of pseudo-holomorphic curves (as in \cite{fooobook2} and others),  we can use a $C^1$ Kuranishi structure 
in place of the $C^{\infty}$ Kuranishi structure.
However when we use the de-Rham  model to obtain a  virtual fundamental 
chain as in \cite{springer}, it seems more natural to use the 
$C^{\infty}$ structure rather than the $C^1$ structure. For such a purpose,
the $C^\infty$ smooth structure arising from Lemma \ref{lem91} will be 
important.
\end{rem}

\section{Double log smooth structure of $\mathcal M^{\rm d}_{k+1,\ell}$}

Motivated by the discussion in the previous section, we revise
the statement of aforementioned lemma as follows.

We take the step of taking another logarithm,
\begin{equation}\label{form912}
S^{\rm d}_j =  {\rm log}\, T^{\rm d}_j,  \qquad
S^{\rm s}_i =  {\rm log}\, T^{\rm s}_i,
\end{equation}
and define
\begin{equation}\label{form92}
s_j = 1/S^{\rm d}_j \in [0,1/\log(-\log c)), 
\quad
\phi_i =  \exp({\sqrt{-1}}\theta_i)/S^{\rm s}_i \in D^2(1/\log(-\log c)).
\end{equation}
Composing  these coordinate changes with the map $\Phi$, we define
\begin{equation}\label{form9393}
\Psi_{s,\phi} : 
\prod_{a\in \mathcal A_{\bf p}^{\rm s} \cup \mathcal A_{\bf p}^{\rm d}} \mathcal V_a 
\times [0,1/\log(-\log c))^{m_{\rm d}} \times (D^2_{\circ}(1/\log(-\log c)))^{m_{\rm s}}
\to \mathcal M^{\rm d}_{k+1,\ell}.
\end{equation}
Now the corrected version of \cite[Lemma 9.1]{const1} is the following.
We call the associated smooth structure \emph{the double log smooth structure}
of $\mathcal M^{\rm d}_{k+1,\ell}$.

\begin{lem}[Double log smooth structure]\label{lem91}
There exists a unique structure of smooth manifold with 
corners on $\mathcal M^{\rm d}_{k+1,\ell}$ 
such that $\Psi_{s,\phi}$ is a diffeomorphism onto its image for each 
${\bf p} \in \mathcal M^{\rm d}_{k+1,\ell}$.\footnote{And for any analytic 
family of coordinates of the nodal points which is used to define $\Phi$.}
\end{lem}
\begin{rem}
\cite[Lemma 9.1]{const1} was used in various references of the present 
authors to obtain a smooth Kuranishi structure 
of the moduli spaces of pseudo-holomorphic curves.
We can replace it by Lemma \ref{lem91}. Then 
all the results and proofs (which use \cite[Lemma 9.1]{const1}) 
go through using the smooth structure of 
Lemma \ref{lem91} in place of \cite[Lemma 9.1]{const1}.
\par
In fact, if $f$ is a smooth function on the parameter  $t_j, \rho_i$ 
then it is a smooth function on the parameter $s_j$, $\phi_i$.
\end{rem}
\begin{rem}
The coordinates $s_j$, $\phi_i$ are used in \cite[(25.1)]{springer}. 
{They} not only define a smooth structure but also define an admissible
smooth structure in the sense defined in \cite[Chapter 25]{springer} 
as the proof below shows.
\end{rem}
\begin{proof}[Proof of Lemma \ref{lem91}]
During this proof we write $\Psi^{\bf p}$ and etc. to clarify that the relevant object is associated to 
${\bf p} \in \mathcal M^{\rm d}_{k+1,\ell}$.
We denote by $\frak v^{\bf p}_a$ an 
element of the first factor of the domain of (\ref{form9393})
for ${\bf p}$.
We allow the positive constants $C_n$ and $c_n$ to vary in the various inequalities appearing in the proof below.
\par
Suppose 
${\bf q} \in {\rm Im}(\Psi^{\bf p})$.
Then $m_{\rm d}^{\bf q} \le m_{\rm d}^{\bf p}$, $m_{\rm s}^{\bf q} \le m_{\rm s}^{\bf p}$.
We may enumerate the marked points so that the $j$-th boundary node 
(resp. the $i$-th interior node) of ${\bf p}$ corresponds to the
$j$-th boundary node (resp. the $i$-th interior node) of ${\bf q}$ for
$j=1,\dots,m_{\rm d}^{\bf q}$ (resp. $i=1,\dots,m_{\rm s}^{\bf q}$).
This enables us to construct a natural (stratawise) smooth embedding 
$
\Phi_{{{\bf p}{\bf q}}}: {\mathcal V}^{\bf q} \hookrightarrow {\mathcal V}^{\bf p}
$
so that we can compare the two  functions  
$T^{\rm d,\bf p}_{j_0} \circ \Phi_{{{\bf p}{\bf q}}}$ and $T^{\rm d,\bf q}_{j_0}$  on $\mathcal V^{\bf q}$.
(See \cite[Proposition 8.27]{foooexp} for the explanation of this coordinate change map
$\Phi_{{{\bf p}{\bf q}}}$.)  This being explicitly said, we will abuse our notation by 
writing $T^{\rm d,\bf p}_{j_0} \circ \Phi_{{{\bf p}{\bf q}}}$ just as 
$T^{\rm d,\bf p}_{j_0}$ dropping the composition by $\Phi_{{{\bf p}{\bf q}}}$
in the following calculations.
\par
Then we can easily prove the next inequalities:
\begin{equation}\label{form94}
\aligned
\left\Vert \nabla^{n-1} \frac{\partial}{\partial T_{j}^{\rm d,\bf q}}(T^{\rm d,\bf p}_{j_0} - T^{\rm d,\bf q}_{j_0})
\right\Vert &\le C_n e^{-c_n T_{j}^{\rm d,\bf q}} \\
\left\Vert \nabla^{n-1} \frac{\partial}{\partial T_{i}^{{\rm s},\bf q}}(T^{\rm d,\bf p}_{j_0} - T^{\rm d,\bf q}_{j_0})
\right\Vert &\le C_n e^{-c_n T_{i}^{\rm s,\bf q}} \\
\left\Vert \nabla^{n-1} \frac{\partial}{\partial \theta_{i}^{\bf q}}(T^{\rm d,\bf p}_{j_0} - T^{\rm d,\bf q}_{j_0})
\right\Vert &\le C_n e^{-c_n T_{i}^{\rm s,\bf q}} 
\endaligned 
\end{equation}
for $j_0 = 1,\dots,m_{\rm d}^{\bf q}$.
Here $\nabla^{n-1}$ ($n \ge 1$)  is the $(n-1)$-th derivatives on the variables  $\frak v^{\bf q}_a$, $T_{j}^{\rm d,\bf q}$,
$T_{i}^{\rm s,\bf q}$, $\theta_{i}^{\bf q}$
and $\Vert \cdot \Vert$ is the $C^0$ norm.
Similar estimates hold for $T^{\rm s,\bf p}_{i_0}$ and $\theta^{\bf p}_{i_0}$.
\par
In fact, to prove the 2nd and 3rd inequalities of (\ref{form94}), we use the fact that 
the functions $\sigma^{\bf p}_i$, $\sigma^{\bf q}_i$ are holomorphic functions defining the same 
divisor and show that  $\sigma^{\bf p}_i/\sigma^{\bf q}_i$ is a nowhere vanishing holomorphic 
function. Then in the same way as \cite[Sublemma 8.29]{foooexp}
we obtain the 2nd and 3rd inequalities of (\ref{form94}).
The 1st inequality is proved in the same way by taking the double 
as in \cite[Section 2]{const1}.
\par 
Moreover 
in the same way we can derive the following:
\begin{equation}\label{form942}
\aligned
\left\Vert \nabla^{n-1} \frac{\partial}{\partial T_{j}^{\rm d,\bf q}}\frac{\partial}{\partial T_{i_0}^{\rm s,\bf q}}
(\theta_{i_0}^{\bf p} - \theta_{i_0}^{\bf q})
\right\Vert &\le C_n e^{-c_n (T_{j}^{\rm d,\bf q}+T_{i_0}^{\rm s,\bf q})} \\
\left\Vert \nabla^{n-1} \frac{\partial}{\partial T_{i}^{{\rm s},\bf q}}\frac{\partial}{\partial T_{i_0}^{\rm s,\bf q}}(\theta_{i_0}^{\bf p} - \theta_{i_0}^{\bf q})
\right\Vert &\le C_n e^{-c_n (T_{i}^{\rm s,\bf q}+T_{i_0}^{\rm s,\bf q})} \\
\left\Vert \nabla^{n-1} \frac{\partial}{\partial \theta_{i}^{\bf q}}\frac{\partial}{\partial T_{i_0}^{\rm s,\bf q}}(\theta_{i_0}^{\bf p} - \theta_{i_0}^{\bf q})
\right\Vert &\le C_n e^{-c_n (T_{i}^{\rm s,\bf q}+T_{i_0}^{\rm s,\bf q})} \blue{.}
\endaligned 
\end{equation}
In the same way we can show similar estimates for higher derivatives on the gluing parameters.

The first formula of (\ref{form94}) applied to the case $j=j_0$ implies that 
$T^{\rm d,\bf p}_{j_0} - T^{\rm d,\bf q}_{j_0}$ is bounded.
Using $ \log(R+c) - \log R = c/R + O(1/R^2)$ and rewriting 
\begin{eqnarray*}
S^{\rm s,\bf p}_{i_0} - S^{\rm s,\bf q}_{i_0} & = & 
\log T^{\rm s,\bf p}_{i_0} - \log T^{\rm s,\bf q}_{i_0} \\
& = & \log \left(T^{\rm s,\bf q}_{i_0} + (T^{\rm s,\bf p}_{i_0} - T^{\rm s,\bf q}_{i_0})\right)
- \log T^{\rm s,\bf q}_{i_0}
\end{eqnarray*}
(for $R := T^{\rm s,\bf q}_{i_0}$, $c: = T^{\rm s,\bf p}_{i_0} - T^{\rm s,\bf q}_{i_0}$),
we can apply
the inequality (\ref{form94}) and a similar inequality for $T^{\rm s,\bf q}_{i_0}$  
to derive 
\begin{equation}\label{form99}
\aligned
\left\Vert \nabla_{\frak v}^{n-1}(S^{\rm s,\bf p}_{i_0} - S^{\rm s,\bf q}_{i_0})
\right\Vert &\le C_n/T^{\rm s,\bf q}_{i_0} \le C_n e^{-c_n S_{i_0}^{\rm s,\bf q}}
\\
\left\Vert \nabla_{\frak v}^{n-1}(S^{\rm d,\bf p}_{j_0} - S^{\rm d,\bf q}_{j_0})
\right\Vert &\le C_n/T^{\rm s,\bf q}_{j_0} \le C_n e^{-c_n S_{j_0}^{\rm d,\bf q}}\blue{.}
\endaligned
\end{equation}
Here $\nabla_{\frak v}^{n-1}$ ($n \ge 1$) is the $(n-1)$-th derivatives on the variables 
 $\frak v^{\bf q}_a$. (In particular, the case $n-1 = 0$ is included.)

We next calculate, for $j\ne j_0$
$$
\aligned
\left\Vert \nabla_{\frak v}^{n-1}\frac{\partial (S^{\rm d,\bf p}_{j_0} - S^{\rm d,\bf q}_{j_0})}{\partial S^{\rm d,\bf q}_{j}} \right\Vert 
&\le  \left\Vert \nabla_{\frak v}^{n-1} \left(\frac{ T^{\rm d,\bf q}_{j_0}}{ T^{\rm d,\bf p}_{j_0}} \frac{\partial }{\partial S^{\rm d,\bf q}_{j}}
\frac{T^{\rm d,\bf p}_{j_0}}{T^{\rm d,\bf q}_{j_0}}\right)\right\Vert 
\\
&\le  \left\Vert \nabla_{\frak v}^{n-1} \left(\frac{ T^{\rm d,\bf q}_{j_0}}{ T^{\rm d,\bf p}_{j_0}}
(T^{\rm d,\bf q}_{j_0})^{-1}\frac{\partial (T^{\rm d,\bf p}_{j_0}- T^{\rm d,\bf q}_{j_0})}{\partial T^{\rm d,{\bf q}}_{j}}\right)\right\Vert 
\Vert T^{\rm d,\bf q}_{j}\Vert \\
&  \le C_n e^{-c_n (S_{j}^{\rm d,\bf q}+S_{j_0}^{\rm d,\bf q})}\blue{.}
\endaligned
$$
For  the second line we used the identity
$T^{\rm d,{\bf q}}_{j}\partial / \partial T^{\rm d,\bf q}_{j} = \partial /\partial S^{\rm d,{\bf q}}_{j}$.
For the third line we used  the inequalities in (\ref{form94}). 
(We also use the fact that $ { T^{\rm d,\bf q}_{j_0}}/{ T^{\rm d,\bf p}_{j_0}}$ is bounded 
together with its $\frak v^{\bf q}$ derivatives.)
\par
For  $j =  j_0$, we have an extra term
$$
\left\Vert \nabla_{\frak v}^{n-1} \left(\frac{ T^{\rm d,\bf q}_{j_0}}{ T^{\rm d,\bf p}_{j_0}}({T^{\rm d,\bf p}_{j_0}})^{-2}
(T^{\rm d,\bf p}_{j_0}- T^{\rm d,\bf q}_{j_0})\right)\right\Vert \Vert T^{\rm d,\bf q}_{j_0}\Vert
$$
in the second line.
It can be estimated also by  $C_n e^{-c_n S_{j_0}^{\rm d,\bf q}}$.

We can estimate $S^{\rm s,\bf q}_{i}$ and $\theta^{\bf q}_{i}$ derivatives in the same way. The estimates of higher derivatives are similar.

Furthermore we can estimate derivatives of $S^{\rm s,\bf p}_{i_0} - S^{\rm s,\bf q}_{i_0}$
in the same way.
Thus we have
\begin{equation}\label{concl1}
\aligned
\left\Vert \nabla^{\prime,n-1}\frac{\partial (S^{\rm d,\bf p}_{j_0} - S^{\rm d,\bf q}_{j_0})}{\partial S^{\rm d,\bf q}_{j}}
\right\Vert &\le C_n e^{-c_n (S_{j}^{\rm d,\bf q}+S_{j_0}^{\rm d,\bf q})} \\
\left\Vert \nabla^{\prime,n-1}\frac{\partial (S^{\rm d,\bf p}_{j_0} - S^{\rm d,\bf q}_{j_0})}{\partial S^{\rm s,\bf q}_{i}}
\right\Vert &\le C_n e^{-c_n (S_{i}^{\rm s,\bf q}+S_{j_0}^{\rm d,\bf q})} \\
\left\Vert \nabla^{\prime,n-1}\frac{\partial (S^{\rm d,\bf p}_{j_0} - S^{\rm d,\bf q}_{j_0})}{\partial \theta^{\bf q}_{i}}
\right\Vert &\le C_n e^{-c_n (S_{i}^{\rm s,\bf q}+S_{j_0}^{\rm d,\bf q})}
\endaligned 
\end{equation}
and
\begin{equation}\label{concl0}
\aligned
\left\Vert \nabla^{\prime,n-1}\frac{\partial (S^{\rm s,\bf p}_{i_0} - S^{\rm s,\bf q}_{i_0})}{\partial S^{\rm d,\bf q}_{j}}
\right\Vert &\le C_n e^{-c_n (S_{j}^{\rm d,\bf q}+S_{i_0}^{\rm s,\bf q})} \\
\left\Vert \nabla^{\prime,n-1}\frac{\partial (S^{\rm s,\bf p}_{i_0} - S^{\rm s,\bf q}_{i_0})}{\partial S^{\rm s,\bf q}_{i}}
\right\Vert &\le C_n e^{-c_n (S_{i}^{\rm s,\bf q}+S_{i_0}^{\rm s,\bf q})} \\
\left\Vert \nabla^{\prime,n-1}\frac{\partial (S^{\rm s,\bf p}_{i_0} - S^{\rm d,\bf s}_{i_0})}{\partial \theta^{\bf q}_{i}}
\right\Vert &\le C_n e^{-c_n (S_{i}^{\rm s,\bf q}+S_{i_0}^{\rm s,\bf q})}\blue{.}
\endaligned 
\end{equation}
Here $\nabla^{\prime,n-1}$ is the $(n-1)$-th derivatives on the variables  $\frak v^{\bf q}_a$, $S_{j}^{\rm d,\bf q}$,
$S_{i}^{\rm s,\bf q}$, $\theta_{i}^{\bf q}$. (Here again 
$n \ge 1$ is assumed and so the case $n-1 = 0$ is included.)
{We can show a similar estimate for  higher derivatives of $S^{\rm d,\bf q}_{j}$, 
${S^{\rm s,\bf q}_{i}}$, ${\theta^{\bf q}_{i}}$ in the same way.}

We next study $\theta_{i_{{0}}}^{\bf q}$.
We put
$$
f_{i_{{0}}} = \lim_{T_{i_{{0}}}^{\rm s,{\bf q}} \to \infty} (\theta_{i_{{0}}}^{{\bf p}} - \theta_{i_{{0}}}^{{\bf q}}).
$$
An analogue of (\ref{form94}) with $T_{j_0}^{{\rm d},\bf q}$ replaced by $\theta_{i_{{0}}}^{\bf q}$
implies that the right hand side limit exists which is smooth on the variables  
$\frak v^{\bf q}_a$, $T_{j}^{\rm d,\bf q}$
$T_{i}^{\rm s,\bf q}$, $\theta_{i}^{\bf q}$. 
{The function $f_{i_{0}}$ is independent of $T_{i_0}^{\rm s,\bf q}$, $\theta_{i_{0}}^{\bf q}$ .}
Then we can use (\ref{form942}) 
to derive the following collection of inequalities:
\begin{equation}\label{concl3}
\aligned
\left\Vert \nabla^{\prime,n-1}\frac{\partial f_{i_0} }{\partial S^{\rm d,\bf q}_{j}}
\right\Vert &\le C_n e^{-c_n S_{j}^{\rm d,\bf q}} \\
\left\Vert \nabla^{\prime,n-1}\frac{\partial f_{i_0} }{\partial S^{\rm s,\bf q}_{i}}
\right\Vert &\le C_n e^{-c_n S_{i}^{\rm s,\bf q}} \\
\left\Vert \nabla^{\prime,n-1}\frac{\partial f_{i_0}}{\partial \theta^{\bf q}_{i}}
\right\Vert &\le C_n e^{-c_n S_{i}^{\rm s,\bf q}}
\endaligned
\end{equation}
and
\begin{equation}\label{concl4}
\aligned
\left\Vert \nabla^{\prime,n-1}\frac{\partial (\theta_{i_0}^{{\bf p}} - \theta_{i_0}^{{\bf q}}-f_{i_0})}{\partial S^{\rm d,\bf q}_{j}}
\right\Vert &\le C_n e^{-c_n (S_{j}^{\rm d,\bf q}+S_{i_0}^{\rm s,\bf q})} \\
\left\Vert \nabla^{\prime,n-1}\frac{\partial (\theta_{i_0}^{{\bf p}} - \theta_{i_0}^{{\bf q}}-f_{i_0})}{\partial S^{\rm s,\bf q}_{i}}
\right\Vert &\le C_n e^{-c_n (S_{i}^{\rm s,\bf q}+S_{i_0}^{\rm s,\bf q})} \\
\left\Vert \nabla^{\prime,n-1}\frac{\partial (\theta_{i_0}^{{\bf p}} - \theta_{i_0}^{{\bf q}}-f_{i_0})}{\partial \theta^{\bf q}_{i}}
\right\Vert &\le C_n e^{-c_n (S_{i}^{\rm s,\bf q}+S_{i_0}^{\rm s,\bf q})}.
\endaligned 
\end{equation}

\par
Furthermore for each $n \geq 1$,  we can combine
(\ref{form99}), (\ref{concl1}),  (\ref{concl0}), (\ref{concl3}), (\ref{concl4}) 
{(and similar estimates for higher derivatives)} to derive
\begin{equation}\label{form100}
\aligned
\left\Vert \nabla^{\prime,n-1}(s^{\bf p}_{j_0} - s^{\bf q}_{j_0})
\right\Vert & \le C_n e^{-c_n / s_{j_0}^{\bf q}}
\\
\left\Vert \nabla^{\prime,n-1}(\phi^{\bf p}_{i_0} - e^{\sqrt{-1}f_{i_0}}\phi^{\bf q}_{i_0})
\right\Vert &\le C_n e^{-c_n / \vert \phi^{\bf q}_{i_0}\vert}.
\endaligned
\end{equation}
($n \ge 1$.)
Finally it follows from (\ref{form100}) that
the map $(\frak v^{\bf q},(s^{\rm d,\bf q}_{j}),(\phi^{\bf q}_{i})) 
\mapsto (\frak v^{\bf p},(s^{\rm d,\bf p}_{j}),(\phi^{\bf p}_{i}))$
is a diffeomorphism (including the point where some of the coordinates are zero).\footnote{
It is easy to see that $\frak v^{\bf p}$ depends smoothly on 
$(\frak v^{\bf q},(s^{\rm d,\bf q}_{j}),(\phi^{\bf q}_{i}))$.}
This implies that the coordinate change map is smooth and so follows the lemma.
\end{proof}

\bibliographystyle{amsalpha}

\end{document}